\documentclass{article}
\usepackage{graphicx} 
\usepackage{tikz-cd}
\usepackage[utf8]{inputenc}
\usepackage{amsmath}
\usepackage{mathtools}
\usepackage{amssymb}
\usepackage{amsthm}
\newtheorem{thm}{Theorem}
\newtheorem{prop}{Proposition}

\newtheorem{cor}{Corollary}
\newcommand{\Hom}{\textrm{Hom}}
\newcommand{\Stab}{\textrm{Stab}}
\newtheorem{defn}{Definition}

\newtheorem{qn}{Question}
\newtheorem{lem}{Lemma}

\newcommand{\Aut}{\textrm{Aut}}
\newcommand{\Out}{\textrm{Out}}
\newcommand{\Mod}{\textrm{Mod}}
\newcommand{\Ac}{\mathcal{A}}
\newcommand{\Lf}{\mathfrak{l}}

\newcommand{\II}{\mathcal{I}}

\newcommand{\LL}{\mathcal{L}}
\newcommand{\PP}{\mathbb{P}}
\newcommand{\OO}{\mathcal{O}}

\newcommand{\ZZ}{\mathbb{Z}}
\newcommand{\QQ}{\mathbb{Q}}
\newcommand{\CC}{\mathbb{C}}
\newcommand{\DD}{\mathbb{D}}

\title{A $\pi_1$ obstruction to having finite index monodromy and an unusual subgroup of infinite index in $\Mod(\Sigma_g)$}
\author{Ishan Banerjee }
\date{April 2023}

\begin{document}

\maketitle
\begin{abstract}
    Let $X$ be an algebraic surface with $\LL$ an ample line bundle on $X$.
    Let $\Gamma(X, \LL)$ be the \emph{geometric monodromy} group a
    associated to family of nonsingular curves  in $X$ that are zero loci of sections of  $\LL$. We provide obstructions to $\Gamma(X, \LL)$ being finite index in the mapping class group. We also show that for any $k \ge 0$, the image of  monodromy  is  finite index in appropriate subgroups of the quotient of the mapping class group by the $k$th term of the Johnson filtration assuming that $\LL$ is sufficiently ample. This enables us to construct several subgroups of the mapping class group with unusual properties, in some cases providing the first examples of subgroups with those properties.

\end{abstract}
\section{Introduction}
Let $X$ be a smooth complex projective surface. Let $\LL$ be an ample line bundle on $X$. Let $H^0(X, \LL)$ be the space of global sections of $\LL$. Let $\PP H^0(X, \LL)$ be its projective space. Since two sections of $\LL$ are scalar multiples of each other if and only if their zero loci are the same, we may identify $\PP H^0(X, \LL)$ with the space of complex curves in $X$ which arise as zeroes of sections of $\LL$.

Let $$ \bar E (X, \LL) = \{ (p,f) \in X \times \PP H^0(X, \LL) | f(p) =0 \}.$$ Let $\pi: \bar E (X, \LL) \to \PP H^0(X, \LL)$  be the projection map. The fibres of $\pi$ are curves in $X$.

Let $$U(X, \LL) = \{   f \in \PP H^0(X, \LL) | \not \exists p \in X,  f(p) =0 \textrm{ and } df(p) =0\}.$$ The space $U(X, \LL)$ parametrizes \emph{nonsingular} curves in $X$
that are zero loci of sections of $\LL$. It is an open subset of $\PP H^0(X, \LL).$
Let $$E(X, \LL) = \pi^{-1}(U(\LL)).$$ The map $\pi: \bar E (X, \LL) \to \PP H^0(X, \LL)$ restricts to a map $E(X, \LL) \to U(X, \LL)$. It is a proper submersion, so by Ehresmann's fibration theorem it is a $C^{\infty}$ fibre bundle. Let $f\in U(X, \LL)$, and let $C \subseteq X$ be the zero locus of $f$. Associated to this fiber bundle is a monodromy representation $ \rho: \pi_1(U(X, \LL),f) \to \Mod (C)$, where $\Mod(C)$ is the mapping class group of $C$. We will denote the image of $\rho$ by $\Gamma(X, \LL)$.

In the case when $X$ is a toric surface the group $\Gamma(X, \LL)$ is completely understood by work of Salter and Calderon- Salter. This is done in the papers \cite{S} and \cite{CaS}. However a toric surface is incredibly special and part of the proof relies on constructing families of vanishing cycles in an extremely combinatorial fashion (called A- cycles and B-cycles in those paper) that does not seem likely to generalise to other kids of surfaces. 

In this paper we will prove some results about the group $\Gamma(X,\LL)$ under various \emph{topological} hypotheses on $X$, specifically hypotheses on $\pi_1(X)$. Our results will be far weaker than those in \cite{S} and \cite{CaS}. However the class of surfaces that we deal with is far larger.

Our main theorems are as follows:

\begin{thm}\label{maincontain}
    Let $i: C \to X$ be the inclusion map. Let $i_*: \pi_1 (C) \to \pi_1(X)$ denote the induced map on $\pi_1$.
    The group $\Gamma (X, \LL)$ is contained in the group $\Ac[i_*]$ where $\Ac[i_*] \subseteq \Mod(C)$ is the subgroup generated by Dehn twists about nonseparating curves in $\ker i_*$. If $\pi_1 X$ is torsion-free and infinite, or if $H^1(X, \QQ) \neq 0$, $\Ac[i_*]$ is of infinite index in $\Mod(C)$.
\end{thm}
The reader may be concerned that it may be the case that $\pi_1 X$ is infinite and the map  $i_*$ is $0$, in which case $\Ac[i_*]$ would be all of $\Mod(C)$.
We note here that by the Lefschetz hyperplane theorem $i_*$ is surjective, so this situation can not arise.

\begin{thm}\label{mainsurject}
    Let $k> 0, r \gg k$.  Let $C$ be the zero locus of a section of $\LL^{\otimes r}$ in $X$. Let $i : C \to X$ be the inclusion. Let $V = i^*(H^1(X)) \subseteq H^1(C)$. Let $\Ac[V] \subseteq \Mod(C)$ be the subgroup generated by Dehn twists about nonseparting simple closed curves in $C$ whose associated homology class in $H^1(C)$ is in $V^{\perp}$. Let  $J_k \subseteq \Mod(C)$ denote the $k$th term of the Johnson filtration. Then  $$\Gamma(X, \LL^{\otimes r})/ J_k \cap \Gamma(X, \LL^{\otimes r})$$ is a finite index subgroup of $\Ac[V] / J_k \cap \Ac[V].$

\end{thm}

There is a tension between the results of Theorems \ref{maincontain} and \ref{mainsurject}. For instance, suppose $\pi_1 X$ is infinite and torsion free and $H^1(X, \QQ) = 0$. In these cases the group $\Gamma(X, \LL^{\otimes r})$ is infinite index in $\Mod(C)$ but for $r \gg k$ , the group is finite index in $\Mod(C)/ J_k$. Thus the fact that $\Gamma(X, \LL^{\otimes r})$ is infinite index must involve it not containing lots of elements deep in the Johnson filtration.

In order to establish Theorem \ref{mainsurject} we establish some of the properties of the higher Johnson-Morita homomorphisms when restricted to $\Ac[V]$. This is covered in Section 5 . We believe this to be of independent interest.

As a result of the above Theorems and also Theorem \ref{thmtechsurj}, we construct subgroups of $\Mod(C)$
with unusual properties.

\begin{cor}\label{cor1}
    Let $k >0$. Let $X$ be an smooth projective surface such that $\pi_1 X$ is a torsion free lattice in $Sp_4(\ZZ)$. Let $\LL$ be an ample line bundle on $X$. Let $r\gg k$. Let $C$ be the zero locus of a section of $\LL^{\otimes r}$. Then, 
    \begin{enumerate}
        \item $\Gamma(X, \LL^{\otimes r})$ is infinite index in $\Mod(C)$.
        \item $\Gamma (X, \LL^{\otimes r}) / \Gamma (X, \LL^{\otimes r}) J_k$ is finite index in $ \Mod(C) / J_k$.
    \end{enumerate}
\end{cor}

\begin{cor}\label{cor2}
    Let $k >0$. Let $X$ be an smooth projective surface such that $\pi_1 X$ is the product of a lattice in $Sp_4(\ZZ)$ and $\ZZ^{2m}$. Let $\LL$ be an ample line bundle on $X$. Let $r\gg k$. Let $C$ be the zero locus of a section of $\LL^{\otimes r}$. Let $V = i^*(H^1(X)) \subseteq H^1(C)$ Then for $r\gg k$, 
    \begin{enumerate}
        \item $\Gamma(X, \LL^{\otimes r})$ is infinite index in $\Ac[V]$.
        \item $\Gamma (X, \LL^{\otimes r}) / \Gamma (X, \LL^{\otimes r}) \cap J_k$ is finite index in $ \Ac[V] / \Ac[V] \cap J_k$.
    \end{enumerate}
\end{cor}

We note that it is well-known that surfaces with such fundamental groups exist.

We note that in the above Theorems and Corollaries, we could have instead of taking a sufficiently large power of a fixed ample line bundle, taken a line bundle that is sufficiently ample in some appropriate sense, i.e. some fixed cohomology groups vanish. Essentially all one has to do is to make assumptions on $\LL$ that enable us to prove analogues of Proposition \ref{embedding}.

\subsection{Vanishing cycles and their relation to monodromy}
In this section we will provide context and motivation for our results.
Suppose that we have a $C^{\infty}$ map $f : \DD \to \PP (H^0(X, \LL))$ such that $f(\DD \setminus \{0\}) \subseteq U(X, \LL)$.

Suppose furthermore that  $f(1) = C$ and that $f^* \pi$ defines a Lefschetz fibration over $\DD$, with the fiber over $0$, $C_0$ being the only singular fiber. We assume $ C_0$ has a unique nodal point.

In this case the theory of Lefschetz fibrations tells us that the topology of the family $f^*\pi$ may be described in terms of a  simple closed curve $\alpha \subseteq C$ that is 'pinched' under parallel transport to the nodal fiber $C_0$. Such an
$\alpha$ is called a vanishing cycle.
A question of Donaldson \cite{D} asks which  nonseparating simple closed curves in $C$
arise as vanishing cycles.

This question is related to (and is in fact almost equivalent to) the problem of determining the monodromy group $\Gamma(X, \LL)$ associated to $\LL$.

In our situation, it is well known that the monodromy group $\Gamma(X, \LL)$ is generated by Dehn twists about vanishing cycles. More precisely we have the following propositions.
\begin{prop} \label{genbydeh}
Let $X, \LL, C$ be as above. Assume $\LL$ is very ample. Suppose that a nonseparating simple closed curve $\alpha \subseteq C$ is a vanishing cycle. Then the Dehn twist $T_\alpha \in \Gamma(X, \LL)$. Furthermore, $\Gamma(X, \LL)$ is generated by Dehn twists about nonseparating vanishing cycles.
\end{prop}

\begin{prop}\label{vcycmon}
    Let $X, \LL, C$ be as above.  Assume $\LL$ is very ample. Let $\Delta$ be the set of nonseparating simple closed curves arising as vanishing cycles.
    Then if $\phi \in \Gamma(X,\LL)$ and $\delta \in \Delta$, $\phi(\delta) \in \Delta$. Thus, $\Gamma(X, \LL)$ acts on $\Delta$. Furthermore this action is transitive.
\end{prop}
A reference for the above two Propositions may be found in \cite{Voi} Section 3.2.3, note that the reference only deals with homology classes of vanishing cycles, but the proofs go through to establish the above propositions after the obvious adjustments. 

While we will not prove Propositions \ref{genbydeh} and \ref{vcycmon} here, we will explain some of the ideas that go into the proofs.

It is a general fact that  $\pi_1(\PP^n \setminus D,x)$ where $D$ is some irreducible hypersurface  and $x \in \PP^n \setminus D$ is generated by meridianal loops that are all conjugate to each other. A meridianal loop is defined as follows: 

\begin{enumerate}
    \item Let $y$ be a  point in the smooth  locus of $D$.
    \item Let $\gamma$ be a path starting at $x$ and ending at $y$ that lies in $\PP^n \setminus D$ except at the endpoint.
    \item The meridianal loop $M_\gamma$ is the loop that starts at $x$, follows $\gamma$ till it reaches a small neighbourhood of $y$, goes counterclockwise around the divisor $D$ and returns along $\gamma^{-1}$. 
\end{enumerate}

In our case, one can argue that the monodromy of these meridianal loops are Dehn twists about nonseparating simple closed curves in $C$. The fact that the meridianal loops are Dehn twists about \emph{nonseparating} curves is a consequence of the fact that a smooth point in $\Sigma(\LL)$ defines an irreducible nodal curve. The fact that these meridians are all conjugate leads to the transitivity of the group $\Gamma(X, \LL)$ on $\Delta$.

As a result of the above two Propositions, every cycle is a vanishing cycle if and only if $\Gamma(X,\LL)= \Mod(C)$.
\subsection{Existing work and a general question}

As mentioned earlier, in the case when $X$ is a toric surface and $\LL$ is an ample line bundle, Salter \cite{S} proves that the monodromy group is either hyperelliptic or finite index in the entire mapping class group. Later work of Calderon-Salter \cite{CaS} proves that the monodromy group is either hyperelliptic or the subgroup of the mapping class group preserving a higher spin structure( see \cite{S}  for further details on what a higher spin structure is). Furthermore if we fix an ample line bundle $\LL$, and look at $\Gamma(X, \LL^{\otimes r})$ for $r$ sufficiently large, then the monodromy group is not hyperelliptic, and hence is finite index. This leads us to the following question.
 \begin{qn}\label{finiteindexqn}
Let X be a smooth projective complex surface. Let $\LL$ be an ample line bundle. Is $\Gamma(X, \LL^{\otimes r})$  finite index for $r$ sufficiently large?
 \end{qn}
 Theorem \ref{maincontain} asserts that if $H^1(X,\QQ) \neq 0$ or if $\pi_1 (X)$ is torsionfree and infinite, then the answer is no. Theorem \ref{mainsurject} asserts that while the answer to the above question is no, a weakening of the question has a positive answer, namely  $\Gamma(X, \LL^{\otimes r} )$ surjects onto a finite index subgroup of the image of $\Ac(V)$ in the quotient of the mapping class group by a term of the Johnson filtration.

Even in the case when $X$ is simply connected or has finite fundamental group, Theorem \ref{maincontain} and Theorem \ref{mainsurject} shut down one avenue of attack on Question \ref{finiteindexqn}, i.e. to prove that the monodromy groups $\Gamma(X, \LL)$ are of finite index. Namely one could hope that subgroups satisfying properties akin to those in Theorem \ref{mainsurject}  and that of Lemma \ref{large} are automatically finite index in the mapping class group, indeed this was the hope of the author going into this project. One could imagine that if a subgroup of the mapping class group contains subgroups that are mapping class groups of sufficiently large subsurfaces of $C$ and contain enough Dehn twists to virtually surject onto $Sp(H_1(C)),$ then the subgroup would be forced to be finite index. However this is simply not the case.

\subsection{Acknowledgements}
I would like to thank Aaron Calderon and Nick Salter for helpful discussions about spin and framed mapping class groups. I'd like to thank Nick Salter for going through an early draft of this paper. I'd like to thank Peter Huxford for helpful comments. 
\section{Containment}
In this section we will establish that the monodromy group $\Gamma (X, \LL)$ is contained in $\Ac[i_*]$. 

We will first study stabilizer subgroups.

\begin{defn}
  Let $G$ be a group. Let $\phi: \pi_1(C) \to G$ be a surjective homomorphism. Let $S$ be the set of conjugacy classes of surjective homomorphisms from $\pi_1(C)$ to G. 
  There is a right action of $\Mod(C)$ on $S$ arising from precomposition. 
  
  We define $\Stab[\phi]$ to be the stabilizer subgroup of $\phi$ under this action.
\end{defn}

The following proposition characterises the Dehn twists about non-separating simple closed curves that lie in $\Stab[\phi]$.

    \begin{prop}\label{inker}
        Let $\alpha \subseteq \Sigma_g$ be a non-separating simple closed curve, disjoint from a basepoint $*$.  Let $T_{\alpha}$ denote the Dehn twist about $\alpha$. Then $T_{\alpha}$ stabilizes $\phi$ up to conjugacy if and only if $\alpha \in \ker \phi$.
    \end{prop}
\begin{proof}
    Suppose $\alpha \in \ker \phi$. Then one can explicitly  find a model of $\Sigma_g$ with a basepoint $*$, and a geometric basis of curves $\{a_1,b_1, \dots a_g, b_g\}$ passing through $*$ such that $\alpha$ is freely isotopic to $a_1$. The element $a_1$ is mapped to $1$ under $\phi$.  Consider an explicit representative curve $\bar \alpha$ of $\alpha$ such that 
    \begin{enumerate}
        \item $\bar{\alpha}$ is disjoint from $*$.
        \item $\bar \alpha$ is disjoint from $\{a_1, \dots, a_g, b_2, \dots b_g\}.$
        \item $\bar \alpha$ intersects $b_1$ once transversally.
    \end{enumerate}.  
    Then $T_{\bar \alpha}(a_i) = a_i$ for all $i$, $T_{\bar \alpha}(b_i) = b_i$ for $i \neq 1$ and $T_{\bar \alpha} (b_1) b_1^{-1}$ is conjugate to $a_1$. Hence, $ \phi\circ T_{\bar \alpha}  = \phi$ which is what we wanted to prove.

    Suppose $T_{\alpha}$ stabilizes $\phi$ up to conjugacy. Then there exists a lift $T$ of $T_{\alpha}$ in $\Mod(\Sigma_g, *)$( the mapping class group of $\Sigma_g$ preserving the marked point $*$) that stabilizes $\phi$ exactly, i.e. $\phi \circ T  = i_*$. However any lift of $T_{\alpha}$ must be of the form $T_{\bar \alpha}$ where $\bar \alpha$ is a simple closed curve, disjoint from $*$, freely (i.e. without necessarily respecting the basepoint $*$) isotopic to $\alpha$.
         But then we may again choose a geometric basis of $\Sigma_g$, $\{a_1, b_1, \dots, a_g, b_g\}$ such that:
         \begin{enumerate}
        \item $\bar \alpha$ is disjoint from $\{a_1, \dots, a_g, b_2, \dots b_g\}.$
        \item $\bar \alpha$ intersects $b_1$ once.
        \item $\bar \alpha$ is freely isotopic to $a_1$.
    \end{enumerate}
Then $T_{\bar \alpha} (b_1) b_1^{-1}$ is conjugate to $a_1$ and hence $\phi(T_{\bar \alpha} (b_1)) = \phi(b_1)$ if and only if $a_1 \in \ker i_*$ which is what we wanted to prove.
    
\end{proof}

Now we establish that our monodromy group lies in the stabilizer subgroup.

\begin{prop}\label{gammainstab}
    Let $X$ be a smooth projective surface, $\LL$ an ample line bundle, $C$ a smooth zero locus of a section of $\LL$. Let $i: C \to X$ denote the inclusion map. Let $\Gamma(X, \LL) \subseteq \Mod(C)$ denote the monodromy group. Then $\Gamma(X,\LL) \subseteq \Stab[i_*]$, and furthermore $\Gamma(X,\LL) \subseteq \Ac[i_*]$.
\end{prop}
\begin{proof}
        Let $\phi \in \Gamma(X, \LL)$. Let $i: C \to X$ be the inclusion map. Then $i$ is homotopic to $i \circ \phi$ essentially by construction, the parallel transport involved in defining the monodromy gives a homotopy between the two maps. Hence they induce the same homomorphism between $\pi_1(C)$ and $\pi_1(X)$ up to conjugation (the fact that we need to deal with homomorphisms up to conjugation is due to the fact that the homotopy need not preserve a base point). 

Hence $\phi \in \Stab[i_*]$.
 If a simple nonseparating closed curve $\alpha$ is such that $T_{\alpha}$ in $\Stab[i]$, Then $\alpha \in \ker i_*$  by Proposition \ref{inker}. Hence, $T_\alpha \in \Ac[i]$. Since $\Gamma(X, \LL)$ is generated by such $T_{\alpha}$, $\Gamma(X, \LL) \subseteq \Ac[i]$.
    \end{proof}
\section{Infinite index}
In this section we will establish that under the assumptions mentioned in Theorem \ref{maincontain}, $\Ac[i_*]$ is of infinite index in the mapping class group. 

Let $g >0$. Let $\Sigma_g$ be a Riemann surface of genus $g$. Let $\gamma_1, \dots \gamma_r$ be simple closed curves in $\Sigma_g$, Let $N \subseteq \pi_1(\Sigma_g)$ be the normal subgroup generated by $\gamma_1, \dots \gamma_r$. Let $G = \pi_1(\Sigma_g)/ N$. Let  $\rho: \pi_1(\Sigma_g) \to G$ denote the quotient map. Let $S$ be the set $\Hom(\pi_1(\Sigma_g), G)/ conj$, i.e the set of homomorphisms from $\pi_1(\Sigma_g)$ to $G$, up to conjugation. The group $\Mod(\Sigma_g)$ acts on $S$ and $\rho$ is an element of $S$.

\begin{prop}\label{stabinf}
    Suppose that $G \neq 0$ is torsion-free.
    The stabilizer subgroup of $\rho$ in $\Mod(\Sigma_g)$ is of infinite index.
\end{prop}
\begin{proof}
    We first establish that there is some pair of essential simple closed curves $\alpha, \beta $ on $\Sigma_g$, where $\alpha$ and $\beta$ have geometric intersection number 1, with $\alpha \in \ker \rho$ and $\beta \not \in \ker \rho$.  Let $\mathcal{C}$ be the graph whose vertices are essential simple closed curves in $\Sigma_g$ and whose edges consist of curves with geometric intersection number 1. The graph $\mathcal{C}$ is connected. If there are no pairs $\alpha, \beta$ of geometric intersection number 1, with $\alpha \in \ker \rho$ and $\beta \not \in \ker \rho$, then $\mathcal{C}$ would not be connected and $\mathcal{C}$ could be separated into two disjoint subgraphs $\mathcal{C}_1, \mathcal{C}_2$ each  consisting of curves in the kernel of $\rho$ and not in the kernel of $\rho$ respectively.

    Let $\alpha, \beta $ on $\Sigma_g$, where $\alpha$ and $\beta$ have geometric intersection number 1, with $\alpha \in \ker \rho$ and $\beta \not \in \ker \rho$. Let $T_{\beta}$ be the Dehn twist about $\beta$.
    Then $\rho(T_{\beta}^k(\alpha))= \rho (\beta)^k \neq 1 $. Hence no power of $T_{\beta}$ stabilizes $\rho$ as otherwise  $\rho(T_{\beta}^k(\alpha))$ would be conjugate to $\rho(\alpha) =1$. Hence the stabilizer subgroup is not finite index.
\end{proof}

\begin{prop}\label{stabinfh1}
    Let $C, X,\LL, i$ be as above. If $H^1(X,\QQ) \neq 0$, $\Stab[i_*]$ is infinite index in the mapping class group.
\end{prop}
\begin{proof}
Let $\alpha \in H^1(X,\QQ)$ be a nonzero class. The Lefschetz hyperplane theorem implies that $i^* \alpha \neq 0$. Then for any $\phi \in \Stab[i_*]$, $\phi^* (i^*\alpha) = i^*\alpha$.  Let $m i^*\alpha$ be  a multiple of $i^*\alpha$ that is integral.

Consider the surjective map $\Mod(C) \to Sp(H^1(C), \ZZ)$ The above implies that $\Stab[i_*]$ is mapped into the stabilizer subgroup of the vector $m i^*(\alpha)$. But this stabilizer subgroup is not finite index in $Sp(H^1(C,\ZZ))$ and hence $\Stab[i_*]$ is not finite index in $\Mod(C)$.

\end{proof}

With the above Propositions we can now prove Theorem \ref{maincontain}. 
\begin{proof}[Proof of Theorem \ref{maincontain}]

    The group $\Gamma(X, \LL)$ is contained in $\Ac[i_*]$ by Proposition \ref{gammainstab}.

    The group $\Ac[i_*] \subseteq \Stab[i_*]$. The group $\Stab[i_*]$in the cases when $H^1(X, \QQ)\neq 0$ or $\pi_1 X$ is torsion free and infinite, is of infinite index in  $\Mod(C)$ by Propositions \ref{stabinf} and \ref{stabinfh1}. Hence $\Ac[i_*]$ is also infinite index.
\end{proof}

\section{$\Gamma(X, \LL^{\otimes r})$  contains framed mapping class groups of subsurfaces}

 \begin{lem}\label{large}
     Let $X, \LL$ be as above. Let $r >0$. Let $C$ be the zero locus a section of $\LL^{\otimes r}.$ Let $i: C \to X$ be the inclusion. Suppose $i^*(H^1(X, \ZZ)) = V\subseteq H^1(C ,\ZZ)$. Then for $r$ sufficiently large:
    \begin{enumerate}
        \item $\Gamma(X, \LL^{\otimes r})$ surjects onto a finite index subgroup of $Sp(V^{\perp},\ZZ)$.
         \item Fix a plane curve singularity $f: \CC^2 \to \CC$. Let $C_0$ be a Milnor fiber of $f$.  Then one can embed $C_0$ as a subsurface of $C$. Let $\Gamma_v$ be the versal deformation monodromy subgroup of $\Mod(C_0)$. Let $\phi: \Mod(C_0) \to \Mod(C)$ be the homomorphism induced by inclusion. Then $\phi(\Gamma_v) \subseteq \Gamma(X, \LL^{\otimes r})$.
     \end{enumerate}
 \end{lem}

For details on what the versal monodromy subgroup is, see \cite{CuS}.
We note that the requirement on $r$ depends on the choice of plane curve singularity $f$, one needs larger vales of $r$ for more complicated singularities.

We will actually prove something stronger than (2)- we will show that that there is a singular curve $C_{sing} \subseteq X$ that is the zero locus of a section of $\LL^{\otimes r}$. The curve $C_{sing}$ will have precisely one singularity that is isomorphic to  $f$ and further more all deformations of $f$ can be realized as deformations of $C_{sing}$ in $X$. This will crucially rely on the description of the versal deformation space in \cite{DH}.
 
 In this section we will establish Lemma \ref{large}. We will start by providing a reference for part (1) of  Lemma \ref{large} as this was already known.

 \begin{proof}[Proof of Part(1)  of Lemma \ref{large}]
     The fact that $\Gamma (X ,\LL^{\otimes r})$ surjects onto a finite index subgroup of $Sp(V^{\perp}, \ZZ)$ is part of Jannsen's thesis (Chapter 3) and follows from Theorem 3.8 of \cite{J}. This establishes (1).
 \end{proof}

The proof of part (2) of Lemma \ref{large} will involve embedding a versal deformation space of an isolated plane curve singularity into $\PP H^0(X, \LL^{\otimes r})$. 

Let $f: \CC^2 \to \CC$ be a polynomial function. Let us assume that $f^{-1}(0)$ has an isolated plane curve singularity at $(0,0)$, i.e. the system of equations $$f(x,y) =0, f_x(x,y) =0, f_y(x,y) =0$$ has an isolated solution at the origin. Let $\CC [[x,y]]$ be the ring of formal power series in $x$ and $y$. Let $I_f$ be the ideal of $\CC [[x,y]]$ generated by $f, f_x$ and $f_y$. Then it is known that $\CC [[x,y]] / I_f$ is finite dimensional.

Let  $X$ be a smooth  complex algebraic surface and with a line bundle $\LL$. Let $p \in X$. The completed local ring of $X$ at $p$ denoted $\bar \OO_{X,p}$ is isomorphic to $\CC[[x,y]]$. Let $U$ be a neighbourhood of $p$ such that $\LL|_{U}$ and $T^*X|_{U}$ are trivial and $U$ is biholomorphic to a disk. Then these trivialisations over $U$ give a map $H^0(U ,\LL) \to \bar \OO_{X,p} $. Furthermore given $f \in H^0(U,\LL)$ we can define $I_f \subseteq \bar \OO_{X,p} $ to be the ideal generated by $(f, f_x, f_y)$, where $x$ and $y$ are analytic local parameters for $X$ at $p$.

\begin{prop}\label{embedding}
    Let $f: \CC^2 \to \CC$  define a plane curve  with an isolated singularity at the origin. Let $X$ be a smooth projective surface, $\LL$ an ample line bundle on $X$. Let $p \in X$. Fix an isomorphism of $\bar \OO_{X,p}$ with $\CC[[x,y]]$. Then for $r$ sufficiently large there exists:
    \begin{enumerate}
        \item $\tilde f \in H^0(X, \LL ^{\otimes r})$,  such that $I_{\tilde f} = I_f$ (under the chosen identification of $\bar \OO_{X,p}$ with $\CC[[x,y]]$) and the zero locus of $\tilde f$ is smooth outside $p$.
       
        \item $g_1, \dots g_r \in H^0(X, \LL^{\otimes r})$ such that the set $\{g_1, \dots, g_r\}$ form a $\CC$ basis of $\bar \OO_{X,p}/ I_{\tilde f} $
    \end{enumerate}  
\end{prop}
\begin{proof}

    We will start with the proof of part (1). Let $k>0$. Let $I_p$ denote the ideal sheaf of $p$. Let $N_k(p)$ denote the  subscheme of $X$ defined by $I_p^k$. By Serre Vanishing, the map $H^0(X, \LL^{\otimes r}) \to H^0(N^k(p), \LL^{\otimes r})$ is surjective, for $r$ big enough. We may consider the restriction, $f|_{N^k(p)}$ as a section of $H^0(N^k(p), \LL^{\otimes r})$. Let $\mathfrak{m}$ denote the maximal ideal of $\CC[[x,y]]$. For $k$ big enough, the ideal $I_f$ contains $\mathfrak{m}^k$  and any $f_1$ such that $f- f_1 \in \mathfrak{m}^k$ satisfies $I_f = I_{f_1}$.

    Consider the variety $$D = \{(p',h) \in X \setminus p \times H^0(X, \LL^{\otimes r}) |  h(p')= 0, dh(p') =0, h|_{N^k(p)} = f|_{N^k(p)}\}.$$ The variety $D$ maps to $$W = \{h \in  H^0(X, \LL^{\otimes r}) | h|_{N^k(p)} = f|_{N^k(p)} \}.$$ It suffices to show that the map is not surjective. This will follow if we can prove that the dimension of the source is less than that of the target. This follows from the following computation, $D$ admits a projection $\phi$ to $X \setminus \{p\}$. The fiber $\phi^{-1}(p')$ consist of sections $h \in H^0(X, \LL^{\otimes r})$ such that $h(p')= 0, dh(p') =0, h|_{N^k(p)} = f|_{N^k(p)}$. Again Serre Vanishing implies that for $r$ big enough, this is an affine space of codimension $\dim X + 1$ in  $W$. Hence the dimension of $D \le \dim W - \dim X -1 + \dim X = \dim W -1$. Hence $D$ does not surject onto $W$.

    For part (2) we may apply Serre Vanishing to conclude that for $r$ sufficiently large, $H^0(X, \LL^{\otimes r})$ surjects onto $\bar \OO_{X,p}/ I_{\tilde f}$, hence we may take a preimage of a basis of $\bar \OO_{X,p}/ I_{\tilde f}$ to obtain our $\{g_1, \dots, g_r\}.$
\end{proof}

Fix a plane curve singularity $f$ and a point $p \in X$. Let $\tilde f, g_1 \dots, g_r$ be as in Proposition \ref{embedding}.
Let $\DD_{\epsilon}$ be the open $\epsilon$ ball in $\CC^r$. Consider the set $$V_{\epsilon} = \{\tilde f + \sum t_i g_i |\, (t_1, \dots, t_r) \in \DD_{\epsilon}\}.$$

\begin{prop}\label{subm}
    Let $B_s$ be open ball of radius $s$ centred at $p$. Let $$E_{V_{\epsilon}}= \{(p, h) \in X \times V_{\epsilon} |\, h(p) =0, p\not \in B_{s}\}.$$ Let $\pi: E_{V_{\epsilon}} \to V_{\epsilon}$ be the  projection. Then for $\epsilon$ small enough $\pi$ is a proper submersion and by Ehresmann's theorem defines a $C^{\infty}$ fiber bundle. Since $V_{\epsilon}$ is contractible, it is a trivial fiber bundle.
    
\end{prop}
\begin{proof}
    Since $\pi$ is obviously proper, it suffices to prove that it is a submersion. Given $f + \sum t_i g_i \in V_{\epsilon}$, we have a section of the bundle $ (\OO_X \oplus T^* X) \oplus \LL^{\otimes r}$ given by  $(f + \sum t_i g_i, df + \sum t_i dg_i)$. It suffices to prove that for $\epsilon$ small enough this section is non vanishing outside $B_s$. This follows by continuity and the fact that $(f,df)$ is nonvanishing outside $B$.
\end{proof}
Propositions \ref{embedding} and \ref{subm} are important because of the description of the versal deformation space given in \cite{DH}, essentially $V_{\epsilon}$ is a versal deformation space of $f$.

 Let $V^{\circ}_{\epsilon} = V \cap U(X, \LL^{\otimes r})$.  Let $C$ be the zero locus of some $h \in V^{\circ}_{\epsilon}$. We have a monodromy homomorphism $\rho: \pi_1(V^{\circ}_{\epsilon}, h) \to \Mod(C).$
\begin{prop} \label{Vmon}
    The monodromy homomorphism $ \rho: \pi_1(V^0_{\epsilon}, h) \to \Mod(C)$ factorises as $$ \pi_1(V^{\circ}_{\epsilon}, h)  \to \Mod(C \cap  \bar B_s , C \cap \partial B_s) \to \Mod(C). $$ Here $\Mod(C \cap  \bar B_s , C \cap \partial B_s)$ denotes the mapping class group of the surface with boundary $C \cap  \bar B_s$ that is the identity map on the boundary. The homomorphism $\Mod(C \cap  \bar B_s , C \cap \partial B_s) \to \Mod(C)$ corresponds to extending by the identity.
\end{prop}
\begin{proof}
    The family $E(X, \LL^{\otimes r})|_{V^{\circ}_{\epsilon}} \to V^{\circ}_{\epsilon}$ is formed by gluing two bundles together, the trivial bundle $E_{V_{\epsilon}}|_{V^{\circ}_{\epsilon}}$ and the nontrivial bundle $\{(f,x) \in V^0_{\epsilon} \times{B_s}| f(x) =0\}.$ Let $$\rho_0:  \pi_1(V^{\circ}_{\epsilon}, h) \to \Mod(C \cap  \bar B_s , C \cap \partial B_s)$$ be the monodromy of the above bundle. Then the  monodromy homomorphism $\rho: \pi_1(V^\circ{}) \to \Mod(C)$ factorises through $\rho_0$ (as is always the case when a bundle is formed by gluing a trivial bundle in this way).

\end{proof}

\begin{prop}\label{versmon}
    The image of the monodromy homomorphism $\pi_1(V^0_{\epsilon}, h)  \to \Mod(C \cap  \bar B_s , C \cap \partial B_s)$ is the monodromy group of the versal deformation space of the plane curve singularity $f$.
\end{prop}
\begin{proof}
    The space $V_{\epsilon}$ is a versal deformation space of $f$ (see for example section 3 of \cite{DH}). Since it is a versal deformation space its monodromy group is the monodromy group of the versal deformation space of the plane curve singularity $f$.
\end{proof}
 We are now in a position to prove part (2) of Lemma \ref{large}
 \begin{proof}[Proof of part(2) of Lemma \ref{large}]
     Part (2) of Lemma \ref{large} follows immediately from 
    Propositions \ref{Vmon} and \ref{versmon}. We note that the Milnor fiber $C_0$ of Lemma \ref{large} is just $C \cap \bar B_s$ of Propositions\ref{Vmon} and \ref{versmon}.
 \end{proof}

\section{The subgroups $\Ac(V)$ and Johnson homomorphisms}
In this section we will discuss the groups $\Ac(V)$ and their image under Johnson-Morita homomorphisms.
Let $C$ be a Riemann surface. Let $H = H^1(C, \QQ)$.
Let $V \subseteq H$ be a symplectic subspace, and $V^{\perp}$ its perpendicular subspace. Let $\rho: \Mod(C) \to Sp(H)$ denote the representation of $\Mod(C)$ arising from the action of mapping classes on cohomology. Since $V \subseteq H$, $Sp(V^{\perp}) \subseteq Sp(H)$

Let $\II(V) = \rho^{-1}(Sp(V^{\perp}))$. Let $\Ac(V)$ be the normal subgroup of $\II(V)$ generated by a single  Dehn twist  about a nonseparating curve $\alpha$, with $[\alpha] \in V^{\perp}$.

\begin{prop}
$\Ac(V)$ is generated by the set of $T_a$ where $a$ is a nonseparating curve, with $[a] \in V^{\perp}$.
\end{prop}
\begin{proof}
    The Proposition immediately follows from the following claim.
    \textbf{Claim:} $\II(V)$ acts transitively on the set of simple closed curves $a$, satisfying $[a] \in V^{\perp}$. 
    
    We will now establish this claim.
    
    Let $a, b$ be non- separating simple closed curves with $[a],[b] \in V^{\perp}.$ Then there exists $g \in Sp(V^{\perp})$ such that $g([a]) = [b]$. Since $\II(V) \to Sp(V^{\perp})$ is surjective, we may find $\tilde g \in \II(V)$ lifting  $g$. Then $\tilde g (a)$ and $b$ are homologous. But by a well known theorem of Johnson
    \cite{Jo} there is an element $h$ in the Torelli group such that $h \tilde g(a) = b$. But $h \tilde g \in \II (V)$. This establishes the claim and hence the Proposition.
\end{proof}

In what follows we will discuss the Malcev completion of groups and the Johnson filtration of the mapping class group. We will be fairly brief, the interested reader is advised to consult \cite{HaJ} and\cite{Jo1} for more details.

Given a group $G$, we will call the $k$th term of its lower central series by $G_k$, with $G = G_0$. We will denote $G/ G_{k}$ by $N_k(G)$ 
Consider the  group $N_k(\pi_1(C))$. It is a nilpotent group and as a result is embedded in a $\QQ$ algebraic nilpotent group called its Malcev completion, which we will denote $L_k(C)$. Let $\Lf_k(C)$ denote the Lie algebra of $L_k(C)$. The functoriality of passing to Malcev completions and taking Lie algebras means that we have a representation $\Mod(C) \to Out(\Lf_k(C))$. Note that $\Lf_k(C)$ is the quotient of the free graded $k$-nilpotent Lie algebra on $H$ by the ideal generated by the symplectic form. 

We recall that the $k^{th}$ term of the Johnson filtration on $\Mod(C)$, $J_k$ is defined by $J_k: = \ker(\Mod(C) \to Out(\Lf_k(C)) )$. 

Let $\Lf_k(V^\perp)$ denote the the free graded $k$ nilpotent Lie algebra on $V^{\perp}$. Let $Aut^s(\Lf_k(V^{\perp}))$ denote the automorphism group of $\Lf_k(V^\perp)$ that fixes the symplectic form $\omega|_{V^{\perp}}$.
 
 Let $\Lf_k(H)$ denote the the free graded $k$ nilpotent Lie algebra on $H$. Let $Aut^s(\Lf_k(H))$ denote the automorphism group of $\Lf_k(H)$ that fixes the symplectic form $\omega$.
There is a homomorphism $\phi: Aut^s(\Lf_k(V^{\perp})) \to Aut^s(\Lf_k(H))$
defined as follows: Since $\Lf_k(H)$ is a free graded nilpotent Lie algebra, to define a map it suffices to define where the generators go. We  define $\phi(g)(h) = h$ for $h \in V$ and $\phi(g)(h)=  g(h)$
 for $h \in V^{\perp}$.

There is a homomorphism $\psi: Aut^s(\Lf_k(H)) \to Aut(\Lf_k(C))$ defined as follows: Since $\Lf_k(C)$ is a quotient of $\Lf_k(H)$ by the ideal generated by $\omega$, we may define $\psi(g)([a]) = [g(a)]$, where $a \in \Lf_k(H) $, $[a]$ is its equivalence class in the quotient. This is well defined as $g(\omega) = \omega$.

Let $C_0\subseteq C$ denote a subsurface of $C$ with 1 boundary component such that the image of $H_1(C_0)$ in $H_1(C)$ is identified with $V^{\perp}$ under Poincare duality. 

\begin{prop}\label{indofch}
    Let $C_0, C_0' \subseteq C$ denote two connected subsurfaces, both with 1 boundary component such that the image of $H_1(C_0)$  and $H^1(C_0')$in $H_1(C)$ is identified with $V^{\perp}$ under Poincare duality. Then there is a diffeomorphism of $C$ taking $C_0$ to $C_0'$, and $\phi$ may be chosen to lie in the Torelli group.
\end{prop}

\begin{proof}
    By the classification of surfaces, one has a self map $C \to C$ that takes $C_0$ to $C_0'$. However as constructed $\phi$ may not lie in the Torelli group. We note that the map on homology induced by $\phi$ preserves $V$ and $V^{\perp}$. Let $\phi_*|_{V} = A_1$ and $\phi_*|{V^{\perp}} = A_2$. Let $(\psi_1, \psi_2)$ be diffeomorphisms of $C_0$ and $C \setminus C_0$ that are the identity on the boundary and whose induced maps on homology are $A_1^{-1}$ and $A_2^{-1}$. Let $\psi$ be the self diffeomorphism of $C$ formed by gluing $\psi_1$ and $\psi_2$. Then one observes that $\phi \circ \psi$ is in the Torelli group and restricts to a diffeomorphism between $C_0$ and $C_0'$.

\end{proof}
Let us fix a point $*$ on the boundary of $C_0$. We have a commutative diagram:
\[
\begin{tikzcd}
    J_k(C_0) \arrow{r} \arrow{d} & \Mod(C, *) \arrow{r} \arrow{d} & \Mod(C) \arrow{d}\\
    \Aut^s (\Lf_k(V^{\perp})) \arrow{r} & \Aut^s(\Lf_k(H))  \arrow{r}& \Out(\Lf_k(C))
\end{tikzcd}
\]

\begin{prop}\label{genbyC0}

       The image of the composite map from $\Mod(C_0) \to Out(\Lf_k(C))$ does not depend on the choice of $C_0$.
    
       The image of $\Ac(V)$ in $Out(\Lf_k(C))$ is precisely the image of the composite map from $\Mod(C_0) \to Out(\Lf_k(C))$.

\end{prop}
\begin{proof}
 The fact that the image of the composite is independent of the choice of $C_0$ is an immediate consequence of Proposition \ref{indofch} . 

 To establish that the two images are the same,  note that the image of $\Mod(C_0)$ in $\Mod(C)$ is contained in $\Ac[V]$ and that given a nonseparating simple closed curve  $a$ with $[a] \in V^{\perp}$, one can choose a $C_0$ satisfying the above properties with $a \in C_0$. Since all $C_0$ have the same image, we have that the image of $\Ac[V]$ is contained in the image of the composite map.


    
\end{proof}

\begin{cor}\label{imageind}
   Let $k \ge 2$. For $\dim V^{\perp} \ll k$, the group $\Ac(V) \cap J_k/ J_{k+1}$ is normally generated in $\II(V)$ by the image of $J_k(C_0)$.
\end{cor}
\begin{proof}
    This immediately follows from the previous Proposition.
\end{proof}

We will also need the following related Proposition.
\begin{prop}\label{AcTor}
    Let  $C$ be a Riemann surface, $\II$ denote its Torelli group. Let $V \subseteq H_1(C, \QQ)$ denote a symplectic subspace. Let $\phi : \II \to \Lambda^3 H_1(C, \ZZ) / H_1(C, \ZZ)$ denote the Johnson homomorphism. 
    Then $$\phi (\II \cap \Ac(V)) \subseteq \Lambda^3 V^{\perp} / V^{\perp}\subseteq  \Lambda^3 H_1(C, \ZZ) / H_1(C, \ZZ).$$
\end{prop}
    \begin{proof}
        Let $C_0\subseteq C$ be a subsurface with one boundary component such that the image of $H_1(C_0, \ZZ) \to H_1(C, \ZZ)$ is $V^{\perp}$. We note that $ \Lambda^3 H_1(C, \ZZ) / H_1(C, \ZZ)$ is isomorphic to $J_1(C) / J_2(C)$. As a result, by Corollary \ref{imageind} the image of $\Ac(V) \cap \II$ in  $ \Lambda^3 H_1(C, \ZZ) / H_1(C, \ZZ)$ agrees with that of $J_1(C_0)$.   However, the naturality of the Johnson homomorphism implies that the Johnson homomorphism $\phi: J_1(C_0) \to J_1(C)/ J_2(C)$ factorises as $$J_1(C_0) \to \Lambda^3 V^{\perp} \to \Lambda^3 H_1(C, \ZZ) / H_1(C, \ZZ), $$ where the first map is the Johnson homomorphism for $C_0$  and the latter map is the map induced by the inclusion $\Lambda^3 V^{\perp} \to \Lambda^3 H_1(C, \ZZ) $ followed by the quotient map.
    \end{proof}
\section{Surjectivity of monodromy groups into quotients}

The main purpose of this section is to prove the following Theorem.

\begin{thm}\label{thmtechsurj}
    Let $X, \LL$ be an algebraic surface, $\LL$ an ample line bundle. Let $r>0$. Let $C$ be a Riemann surface embedded in $X$ as a section of $\LL^{\otimes r}$. Let $i: C\to X$ denote the inclusion. Let $V = i^*(H^1(X, \QQ)).$ Let $k\ge 1$. Then for $r\gg k$, $\Gamma(X, \LL^{\otimes r}) \cap J_k / J_{k+1} $ is finite index in $A(V) \cap J_k / J_{k+1}$.

\end{thm}
\begin{proof}
    The proof will be by the following steps:
    \begin{enumerate}
   
        \item We prove that $\Gamma(X, \LL^{\otimes r})$ contains a framed mapping class group of a large subsurface $C'$ contained in a subsurface of the form $C_0$, where $C_0$ is as in Proposition \ref{genbyC0}.
        \item The image of $\Gamma(X, \LL^{\otimes r})$ in $Sp(V^{\perp})$ is finite index.     Thus  $\Gamma(X, \LL^{\otimes r}) \cap J_k / J_{k+1}$ is preserved under the action of a finite index subgroup of $Sp(V^{\perp})$.
        \item We argue that (1) and (2) imply that $\Gamma(X, \LL^{\otimes r}) \cap J_k / J_{k+1}$ contains a finite index subgroup of the image of $J_k(C_0)$.
        \item We use Proposition\ref{genbyC0} to argue that this then implies the result.
         \end{enumerate}
Part (1) is almost the same as part (2) of  Lemma \ref{large}.  We note that the subsurface $C'$ given to us by Lemma \ref{large} is such that its homology has a basis of vanishing cycles. As a result the map $H^1(X, \ZZ) \to H^1(C', \ZZ)$ is 0. Hence we may extend $C'$ to a maximal $C_0$ such that the pullback $H^1(X,\ZZ) \to H^1(C_0, \ZZ)$ is zero. This establishes (1).
         
To prove (2) note that by Lemma \ref{large}, $\Gamma$ surjects onto a finite index subgroup of  $Sp(V^{\perp})$  say $G$. Hence $\Gamma(X, \LL^{\otimes r}) \cap J_k / J_{k+1}$ is preserved under the action of conjugation by elements in $G$. This establishes (2).

We will establish (3) for $k \ge 2$ and $k =1$ separately. We first deal with the case when $k \ge 2$. When $k \ge 2$, $J_k(C')$ is contained in the framed mapping class group of $C'$.

We note that in the subsurface $C_0$, as long as $g(C') \gg k$ the orbit of the image of $J_k(C')$ under a finite index subgroup of $Sp(V^{\perp})$ is finite index in $J_k(C_0)/ J_{k+1}(C_0).$ This is due to the following reason- by Theorem  of \cite{CP}, the $\Mod(C_0)$ orbits of $J_k(C')$  generates $J_k(C_0)$. This implies that the $Sp(V^{\perp})$ orbits of $J_k(C')/ J_{k} (C') \cap J_{k+1}(C_0)$ generates $J_k(C_0)/ J_{k+1}(C_0)$ and as a result the orbits of $J_k(C')/ J_{k} (C') \cap J_{k+1}(C_0)$ under a finite index subgroup of $Sp(V^{\perp})$  generates a finite index subgroup of  $J_k(C_0)/ J_{k+1}(C_0)$.  This establishes (3) in the case when $k \ge 2$.

When $k =1$, we note that according to \cite{CaS} a framed mapping class group on the surface $C'$ contains the kernel $K(C')$ of the Chillingworth map $\II(C') \to H_1(C')$. This kernel contains an element $\alpha$ such that the composite homomorphism $$ \II(C') \to \Lambda^3 H_1(C', \QQ) /H_1(C, \QQ) \to \Lambda^3 H_1(C) / H_1(C,\QQ)$$ sends $\alpha$ to a nonzero element, here the homomorphism is given by first performing the Johnson homomorphism and composing it with the induced map coming from $H_1(C' \QQ) \to H_1(C,\QQ)$.

The image of such an $\alpha$ in $\Lambda^3 H_1(C) / H_1(C,\QQ)$ is forced to lie in the subgroup $\Lambda^3(V^{\perp})/ V^{\perp}$. However, $$\Lambda^3(V^{\perp})/ V^{\perp} \otimes \QQ$$ is an irreducible representation of $Sp(V^{\perp} \otimes \QQ)$ and as a result, any nontrivial subgroup of $$\Lambda^3(V^{\perp})/ V^{\perp}$$ that is invariant under $Sp(V^{\perp})$ is finite index. As a result $\Gamma \cap J_1(C)$ surjects onto a finite index subgroup of $\Lambda^3(V^{\perp})/ V^{\perp}$. As we have established in Proposition \ref{AcTor} , this implies that $\Gamma \cap J_1(C)/ \Gamma \cap J_2(C)$ is finite index in $\Ac(V) \cap J_1(C)/ \Ac(V) \cap J_2(C).$

Due to (3), $\Gamma(X, \LL^{\otimes r}) \cap J_k / J_{k+1}$ contains a finite index subgroup of the image of $J_k(C_0)$. We then note that $\Ac[V] \cap J_k$ and $J_k(C_0)$ have the same image by Proposition \ref{genbyC0}. This establishes (4).

\end{proof}

We now prove the two corollaries.
\begin{proof}[Proof of Corollary \ref{cor1}]
    This is an immediate consequence of Theorems \ref{maincontain} and \ref{mainsurject}, using the fact that for any lattice $G \subseteq Sp_4(\ZZ)$, $H_1(G, \QQ) =0$ (this is true for any lattice in a higher rank Lie group, by a theorem of Raghunathan \cite{R}). 
\end{proof}

\begin{proof}[Proof of Corollary \ref{cor2}]
    Part (2) is an immediate consequence of the fact that for any lattice $G \subseteq Sp_4(\ZZ)$, $H_1(G, \QQ) =0$ and Theorem \ref{mainsurject}. 

    Let $\gamma$ be a nonseparating simple closed curve such that $[\gamma] \in V^{\perp}$ and $\gamma \not \in \ker i_*$. Then we claim that $T_{\gamma}^k$ is not in $\Stab [i_*]$ and hence not in $\Gamma$ for all $k \in \ZZ \setminus \{0\}$.
    The proof of the claim is as follows, consider some nonseparating simple closed curve $\beta$, intersecting $\gamma$ at one point transversely such that $\beta \in \ker i_*$ (such a $\beta$ always exists by an argument identical to that in Proposition \ref{stabinf}). Then if $T_{\gamma}^k \in \Stab[i_*]$, $i_*(T_{\gamma}^k \beta)$ is conjugate to $i_*(\beta) = 1$. But $i_*(T_{\gamma}^k \beta)$ is conjugate to $i_*(\gamma)^k \neq 1$, since $G$ is torsion-free.

    But the fact that $T_{\gamma}^k$ is never in $\Gamma$ for $k \neq 0$, and $T_{\gamma}^k \in \Ac(V)$ for all $k$ implies that $\Gamma$ is of infinite index in $\Ac(V).$  
\end{proof}

\end{document}